\documentclass[a4paper,12pt,reqno]{amsart}

\usepackage[text={130mm,200mm}]{geometry}

\theoremstyle{plain}
\newtheorem{theorem}{Theorem}
\newtheorem*{theorem*}{Theorem}

\newtheorem*{corollary*}{Corollary}
\newtheorem{lemma}{Lemma}
\newtheorem*{lemma*}{Lemma}

\newtheorem*{proposition*}{Proposition}

\newtheorem*{conjecture*}{Conjecture}
\theoremstyle{definition}

\newtheorem*{definition*}{Definition}
\theoremstyle{remark}
\newtheorem{remark}{Remark}
\newtheorem*{remark*}{Remark}

\begin{document}

\title[Functions with complicated local structure]{On one class of functions with~complicated local structure}
\author{Symon Serbenyuk}
\address{Institute of Mathematics \\
 National Academy of Science of Ukraine \\
  3, Tereschenkivska Str. \\
  Kyiv \\
  01601 \\
  Ukraine}
\email{simon6@ukr.net}





\begin{abstract}

We introduce a class $\Lambda_{s}$ of functions with complicated local structure. Any function from the class belongs to one of three specifically defined types $f^s _k$, $f_+$, and $f^{-1} _+$ or is a specifically defined composition of two or three functions of these types. Differential, integral, fractal and other properties of such functions are investigated. In particular, all functions $f$ from  $\Lambda_{s}$ such that $f(x)\ne x$, $f(x)\ne -\frac{s-1}{s+1}-x$ and $f(x)\ne 1-x$ are nonmonotonic and nowhere differentiable. The Hausdorff--Besicovitch dimension of a plot of each $f \in \Lambda_s$ is equal to $1,$ and the Lebesgue integral of $f$ is equal to $\frac 1 2$. The proof of these statements for the compositions 
and the corresponding proofs for $f^s _k$, $f_+$, and $f^{-1} _+$ are similar.
 
\end{abstract}

\maketitle



\section{Introduction}

Before the nineteenth century due to a number of objective reasons, the mathematicians believe intuitively that any function has derivatives of all orders. With the lapse of time, this idea began to give rise to doubt. Among the first scientists-opponents, we mention Bolzano, Lobachevsky, Dirichlet, and Weierstrass. The latter constructed the example of a continuous nowhere differentiable function in 1871~\cite[pp.~105--108]{Turbin:1992:FMF}. Thereafter, a lot of  mathematicians concentrated their attention on finding the examples and the collections of such functions~\cite{{Pratsevityi:1989:NKP}, {Pratsiovytyi:1998:FPD}, {Pratsiovytyi:2002:FVO}, {Turbin:1992:FMF}}.

The construction of new examples of continuous nowhere differentiable functions was accompanied by the development of new methods of definition of such functions. In the present article, we use a rather simple means to set a whole class of functions with complicated local structure such that almost all functions of the class are everywhere continuous and  nowhere differentiable.


Let $s>1$ be a fixed natural number, and let the set $A=\{0,1,...,s-1\}$ be the alphabet of an s-adic or nega-s-adic number system. The notation $x=~\Delta^{\pm s} _{\alpha_1\alpha_2...\alpha_n...}$ means that $x$ is represented by the s-adic or nega-s-adic number system, i.e.,
$$
x=\sum^{\infty} _{n=1}{\frac{\alpha_n}{s^n}}\equiv \Delta^{ s} _{\alpha_1\alpha_2...\alpha_n...}
$$
or
$$
x=\sum^{\infty} _{n=1}{\frac{(-1)^n\alpha_n}{s^n}}\equiv \Delta^{ -s} _{\alpha_1\alpha_2...\alpha_n...}, \alpha_n\in A.
$$

Let $\Lambda_{s}$ be a class of functions of the type 
\begin{equation}
  \label{eq:1}
 f: x=\Delta^{\pm s} _{\alpha_1\alpha_2...\alpha_n...} \to \Delta^{\pm s} _{\beta_1\beta_2...\beta_n...}=f(x)=y,
\end{equation}
where $\left(\beta_{km+1},\beta_{km+2},...,\beta_{(m+1)k}\right)=\theta\left(\alpha_{km+1},\alpha_{km+2},...,\alpha_{(m+1)k}\right)$, the number $k$ is a fixed positive integer for a specific function $f$, $m=0,1,2,...,$ and
$\theta(\gamma_1,\gamma_2,...,\gamma_k)$ is some function of $k$ variables (it is the bijective correspondence) such that the  set  
$$
A^k=\underbrace{A \times A \times ...\times A}_{k}.
$$
is its domain of definition and range of values.

Each combination $(\gamma_1,\gamma_2,...,\gamma_k)$ of $k$ s-adic or nega-s-adic digits (according to the number representation of the argument of a function $f$) is assigned to the single combination $\theta(\gamma_1,\gamma_2,...,\gamma_k)$ of $k$ s-adic or nega-s-adic digits  (according to the number representation of the value of a function $f$). The combination $\theta(\gamma_1,\gamma_2,...,\gamma_k)$ is assigned to the unique combination $(\gamma^{'} _1,\gamma^{'} _2,...,\gamma^{'} _k)$ that may be not to match with 
$(\gamma_1,\gamma_2,...,\gamma_k)$. The $\theta$ is a bijective function on $A^k$.

 It is clear that any function  $f \in \Lambda_{s}$ is  one of the following functions:
$$
f^s _k, ~~~f_+,~~~ f^{-1} _+,~~~ f_+ \circ f^s _k,~~~ f^s _k \circ f^{-1} _+, ~~~ f_+ \circ f^s _k \circ  f^{-1} _+, 
$$
where
\begin{equation}
\label{form: f^s _k1}
f^s _k\left(\Delta^s _{\alpha_1\alpha_2...\alpha_n...}\right)=  \Delta^{s}_{\beta_{1}\beta_{2}...\beta_{n}...}, 
\end{equation}
$\left(\beta_{km+1},\beta_{km+2},...,\beta_{(m+1)k}\right)=\theta\left(\alpha_{km+1},\alpha_{km+2},...,\alpha_{(m+1)k}\right)$ for $m=0,1,2,...,$ and  some fixed natural number $k$, i.e., 
$$
\left(\beta_{1},\beta_{2},...,\beta_{k}\right)=\theta\left(\alpha_{1},\alpha_{2},...,\alpha_{k}\right),
$$
$$
\left(\beta_{k+1},\beta_{k+2},...,\beta_{2k}\right)=\theta\left(\alpha_{k+1},\alpha_{k+2},...,\alpha_{2k}\right),
$$
$$
.........................
$$
$$
\left(\beta_{km+1},\beta_{km+2},...,\beta_{(m+1)k}\right)=\theta\left(\alpha_{km+1},\alpha_{km+2},...,\alpha_{(m+1)k}\right),
$$
$$
.........................
$$
and
 \begin{equation}
\label{form: f _+1}
f_+\left(\Delta^s _{\alpha_1\alpha_2...\alpha_n...}\right)=\Delta^{-s} _{\alpha_1\alpha_2...\alpha_n...},
\end{equation}
\begin{equation}
\label{form: f^{-1} _+1}
f^{-1} _+\left(\Delta^{-s} _{\alpha_1\alpha_2...\alpha_n...}\right)= \Delta^s _{\alpha_1\alpha_2...\alpha_n...}.
\end{equation}

In \cite{Serbenyuk:2012:OMS}, the following function was considered:
$$
x=\Delta^3 _{\alpha_1\alpha_2...\alpha_n...} \stackrel{f}{\rightarrow} \Delta^{3}_{\varphi(\alpha_{1})\varphi(\alpha_{2})...\varphi(\alpha_{n})...}=f(x)=y,
$$
where $\varphi\left(\alpha_{n}\right)$ is a function defined in terms of the s-adic number system in the following way:

\begin{center}
\begin{tabular}{|c|c|c|c|}
\hline
$\alpha_n$ & $ 0 $& $1$ & $2$\\
\hline
$\varphi (\alpha_n)$ & $0$ & $2$ & $1$\\
\hline
\end{tabular}
\end{center}

The function $f$ from  \cite{Serbenyuk:2012:OMS} is a function  of the $f^3 _1$ type. Now, we present the example of the function $f^2 _2$:
$$
f^2 _2: \Delta^2 _{\alpha_1\alpha_2...\alpha_n...} \rightarrow \Delta^2 _{\beta_1\beta_2...\beta_n...},
$$ 
where $(\beta_{2m+1},\beta_{2(m+1)})=\theta(\alpha_{2m+1},\alpha_{2(m+1)})$, $m=0,1,2,3,...,$ and 
\begin{center}
\begin{tabular}{|c|c|c|c|c|c|}
\hline
$\alpha_{2m+1}\alpha_{2(m+1)}$ & $ 00 $& $01$ & $10$ & $11$\\
\hline
$\beta_{2m+1}\beta_{2(m+1)}$ & $10$ & $11$ & $00$ & $01$\\
\hline
\end{tabular}
\end{center}
 is an example of the $f^2 _2$-type function.

It is obvious that the set of $f^2 _1$ functions consists only of the functions $y=x$ and $y=1-x$ in the binary number system. But a set of $f^2 _2$ functions has the order equal to   $4!$ and includes the functions $y=x$ and  $y=1-x$ as well.

\begin{remark*}
The class of functions $\Lambda_s$ includes the following  linear functions $y=x$, 
$$
x=\Delta^s _{\alpha_1\alpha_2...\alpha_n...}\to \Delta^s _{[s-1-\alpha_1][s-1-\alpha_2]...[s-1-\alpha_n]...}=y=1-x,
$$
$$
x=\Delta^{-s} _{\alpha_1\alpha_2...\alpha_n...}\to \Delta^{-s} _{[s-1-\alpha_1][s-1-\alpha_2]...[s-1-\alpha_n]...}=y=-\frac{s-1}{s+1}-x.
$$
These functions are called \emph{$\Lambda_s$-linear functions.}
\end{remark*}

\section{On the well-posedness of definition of the functions from $\Lambda_{s}$}
\label{sec: correctness}

\begin{lemma}
For any function  $f$  from $\Lambda_{s}$ except for $\Lambda_s$-linear functions, values of function $f$ for different representations of s-adic rational numbers from ~$[0;1]$ (nega-s-adic rational numbers from $[-\frac{s}{s+1};\frac{1}{s+1}]$ respectively) are different.
\end{lemma}
\begin{proof}
Consider an s-adic rational number
$$
x_{(0,n)}=\Delta^s _{\alpha_1\alpha_2...\alpha_{n-1}\alpha_n000...}=\Delta^s _{\alpha_1\alpha_2...\alpha_{n-1}[\alpha_n-1][s-1][s-1][s-1]...}=x_{(s-1,n)}, ~n \in \mathbb Z_0.
$$ 
Let $n = k$. Then the equality  $f^s _k(x_{(0,n)})=f^s _k(x_{(s-1,n)})$ is true for the numbers  $x_{(0,k)}=\Delta^s _{\underbrace{0...01}_{k}000...}$ та $x_{(s-1,k)}=\Delta^s _{\underbrace{0...0}_{k}[s-1][s-1][s-1]...}$  as soon as the system
$$
\left\{
\begin{aligned}
\theta(\underbrace{0,0,...,0,0}_{k}) & =(\underbrace{0,0,...,0,0}_{k}),\\
\theta(\underbrace{s-1,s-1,...,s-1,s-1}_{k}) & =(\underbrace{s-1,s-1,...,s-1,s-1}_{k});
\end{aligned}
\right.\\
$$
is true, and $f^s _k (\Delta^s _{\underbrace{0...01}_{k}000...})=\Delta^s _{\underbrace{0...01}_{k}000...}=\Delta^s _{\underbrace{0...0}_{k}[s-1][s-1][s-1]...}.$ 

In the general case, we have
\begin{align*}
f^s _k (x_{(0,n)})&=f^s _k(\Delta^s _{\alpha_1\alpha_2...\alpha_{n-1}\alpha_n000...})=\Delta^s _{\beta_1\beta_2...\beta_{n-1}\beta_n000...}
\intertext{and}
f^s _k(x_{(s-1,n)})&=f^s _k (\Delta^s _{\alpha_1\alpha_1...\alpha_{n-1}[\alpha_n-1][s-1][s-1][s-1]...})= \\
&=\Delta^s _{\beta_1\beta_2...\beta_{n-1}[\beta_{n}-1][s-1][s-1][s-1]...}.
\end{align*}

In other words, $f^s _k(x_{(s-1,n)})=f^s _k(x_{(0,n)})$, when $f^s _k=x$.

If
$$
\left\{
\begin{aligned}
\theta(\underbrace{0,0,...,0,0}_{k}) & =(\underbrace{s-1,s-1,...,s-1,s-1}_{k}),\\
 (\underbrace{0,0,...,0,0}_{k})& =\theta(\underbrace{s-1,s-1,...,s-1,s-1}_{k}),
\end{aligned}
\right.\\
$$
then the equality 
$$
f^s _k(x_{(s-1,n)})=f^s _k(x_{(0,n)})
$$
 is true, when  $f^s _k=1-x$.

Analogously,
$$
f_{+} (\Delta^s _{\alpha_1\alpha_2...\alpha_{n-1}\alpha_n000...}) \ne f_{+}(\Delta^s _{\alpha_1\alpha_2...\alpha_{n-1}[\alpha_n-1][s-1][s-1][s-1]...}),
$$
and
$$
f^{-1} _{+}(\Delta^{-s} _{\alpha_1\alpha_2...\alpha_{n-1}\alpha_n[s-1]0[s-1]0...}) \ne f^{-1} _{+}(\Delta^{-s} _{\alpha_1\alpha_2...\alpha_{n-1}[\alpha_n-1]0[s-1]0[s-1]0...}).
$$

The proof of the lemma is analogous for the combinations $f_{+} \circ f^s _k $,   $f^s _k \circ f^{-1} _+$, and $f_+ \circ f^s _k \circ  f^{-1} _+$.
\end{proof}

\begin{remark}
From unique representation for each s-adic irrational number from $[0;1]$ it follows that the  function $f^s_k$ is well defined at s-adic irrational  points.

To reach that any function $f \in \Lambda_{s}$ such that $f(x)\ne x$ and $f(x)\ne1-x$ be well-defined on the set of s-adic rational numbers from $[0;1]$, we shall not consider the s-adic representation, which has period $(s-1)$.

Analogously, we shall not consider the nega-s-adic representation, which has period $(0[s-1])$.
\end{remark}

\section{Some properties}
\label{sec: Some properties}

\begin{lemma}
A set of functions  $f^s _k$ with the defined operation ``composition of functions'' is a finite group that has order equal to $\left(s^k\right)!$.
\end{lemma}
\begin{proof}
Let $s>1$ be a positive integer and let $k \in \mathbb N$ be an arbitrary number. The set $A^{k} _{(s)}$ is a set of all samples ordered with reiterations of $k$ numbers from $A=\{0,1,...,s-1\}$. The set $f^s_k$ is a group of permutations of elements from  $A^{k} _{(s)}$. In other words, $f^s _k$ is a symmetric group.
\end{proof}

\begin{lemma}
The function $f \in \Lambda_{s}$ such that $f(x)\ne x$, $f(x)\ne -\frac{s-1}{s+1}-x$ and $f(x)\ne 1-x$ has the following properties:
\begin{enumerate}
\item $f$ reflects  $[0;1]$ or $[-\frac{s}{s+1};\frac{1}{s+1}]$ (according to the number representation of the argument of a function $f$) into one of the segments $[0;1]$ or $[-\frac{s}{s+1};\frac{1}{s+1}]$  without enumerable subset of points (according to the number representation of the value of a function $f$) .
\item the function $f$ is not monotonic on the domain;
\item the function $f$ is not a bijective mapping on the domain.
\end{enumerate}
\end{lemma}
\begin{proof} From \eqref{form: f^s _k1}, it follows the first property for $f$ and the second property for $f^s _k$. 

{\itshape Function $f_+$.}  Let  $x_1=\Delta^s _{\alpha_1\alpha_2...\alpha_n...}$ and $x_2=\Delta^s _{\beta_1\beta_2...\beta_n...}$ be such that $x_1 < x_2$. It is obvious that there exists $n_0$ such that $\alpha_j=\beta_j$ for all $j=\overline{1,n_0-1}$ and $\alpha_{n_0}<\beta_{n_0}$. This implies the system
$$
\left\{
\begin{array}{rcl}
f (x_1)=\Delta^{-s} _{\alpha_1\alpha_2...\alpha_{n_0-1}\alpha_{n_0}...}<&\Delta^{-s} _{\beta_1\beta_2...\beta_{n_0-1}\beta_{n_0}...}=f (x_2), &\text{if $n_0 \equiv 0 \pmod 2$;}\\
f (x_1)=\Delta^{-s} _{\alpha_1\alpha_2...\alpha_{n_0-1}\alpha_{n_0}...}>&\Delta^{-s} _{\beta_1\beta_2...\beta_{n_0-1}\beta_{n_0}...}=f (x_2), &\text{if  $n_0 \equiv 1 \pmod 2$ }.\\
\end{array}
\right.
$$
So, $f_+$ is not monotonic. It is clear that the other functions $f \in \Lambda_{s}$ such that $f(x)\ne x$ and $f(x)\ne 1-x$ are not monotonic as well. 

The $f^s_k$-type function is a bijective mapping on $[0;1]$, when  $y^s _k=x$ or $y^s _k=1-x$. Let us have
$x_1=\Delta^s _{\alpha_1\alpha_2...\alpha_n...}$ and $x_2=\Delta^s _{\beta_1\beta_2...\beta_n...}$ such that  $x_1 \ne x_2$.

Let $y_{1,2}=\Delta^s _{\gamma_1\gamma_2...\gamma_n...}=f^s _k (x_1)=f^s _k (x_2)$ be an s-adic irrational number. Then, for each $m=0,1,...$,
$$
(\gamma_{km+1},\gamma_{km+2},...,\gamma_{(m+1)k})=\theta(\alpha_{km+1},\alpha_{km+2},...,\alpha_{(m+1)k})=\theta(\beta_{km+1},\beta_{km+2},...,\beta_{(m+1)k}).
$$
From the last equalities, it follows that $x_1=x_2$, but this  contradicts the condition.

Let  $y_{1,2}$ be an s-adic rational number, i.e.,
\begin{equation}
\label{a1}
y_{1,2}=\Delta^s _{\gamma_1\gamma_2...\gamma_{n-1}\gamma_n000...}=\Delta^s _{\gamma_1\gamma_2...\gamma_{n-1}[\gamma_n-1][s-1][s-1]...}.
\end{equation}

 It is obvious that $f^s _k$ is not a bijective mapping on the set $M \subset [0;1]$ of rational numbers from $[0;1]$ such that, 
for $M \ni x_1\ne x_2 \in M$, $f^s _k (x_1)=~f^s _k (x_2)=~y_{1,2}$.

It is easy to find the subsets of s-adic rational or nega-s-adic rational numbers such that the functions $ f^{-1} _+$, $ f_+ \circ f^s _k,$ $ f^s _k \circ f^{-1} _+$, $ f_+ \circ f^s _k \circ  f^{-1} _+$ are not bijective mappings on these subsets.
\end{proof}

\begin{lemma}
\label{lm4}
 For each $x \in [0;1]$, the function $ f_+$ satisfies the  equation
\begin{equation}
\label{eq1}
f(x) +f (1-x)=-\frac{s-1}{s+1};
\end{equation}
\end{lemma}
\begin{proof} Let $x=\Delta^s _{\alpha_1\alpha_2...\alpha_n...}$ be an any number from $[0;1]$. Then
\begin{align*}
f_+ (x) +f _+ (1-x)&=f_+ (\Delta^s _{\alpha_1\alpha_2...\alpha_n...})+f_+ (\Delta^s _{[s-1-\alpha_1][s-1-\alpha_2]...[s-1-\alpha_n]...})= \\
&=\Delta^{-s} _{(s-1)}=- \frac{s-1}{s+1}.
\qedhere
\end{align*}
\end{proof}

\begin{lemma}\cite[pp. 13-14, 19]{Pelyuh:1974:VTF} Equation \eqref{eq1} has the solutions 
$$
f(x)=-\frac{s-1}{2(s+1)}+\begin{cases}
h(x-x^2),&\text{ $x<\frac1 2$;}\\
-h(x-x^2),&\text{ $x> \frac1 2$,}
\end{cases}
$$
where $h=\sqrt{\frac{1}{4}\left(\frac{s-1}{s+1}\right)^2-g}$, and $g$ is any function,
or
$$
f(x)=-\frac{s-1}{2(s+1)}+\left(x-\frac1 2\right)\omega\left(x-x^2\right),
$$
where $\omega(x)$ is any function.
\end{lemma}

\begin{lemma}
For each $y \in [-\frac{s}{s+1};\frac{1}{s+1} ]$, the function   $ f^{-1} _+$ satisfies the equation
\begin{equation}
\label{eq2}
f^{-1} (y) +f^{-1}  \left(-\frac{s-1}{s+1}-y\right)=1.
\end{equation}
\end{lemma}
The proof of the lemma is analogous to that of Lemma \ref{lm4}. 

\begin{lemma}\cite[pp. 13-14, 19]{Pelyuh:1974:VTF} The equation  \eqref{eq2} has the solutions 
$$
f^{-1} (y)=\frac1 2+\begin{cases}
h\left(-\frac{s-1}{s+1}y-y^2\right),&\text{ $y<-\frac{s-1}{2(s+1)}$;}\\
-h\left(-\frac{s-1}{s+1}y-y^2\right),&\text{ $y> -\frac{s-1}{2(s+1)}$,}
\end{cases}
~~~\text{where}~~~~~h=\sqrt{\frac{1}{4}-g},
$$
and $g$ is any function or
$$
f^{-1} (y)=\frac1 2 +\left(y+\frac{s-1}{2(s+1)}\right)\omega\left(-\frac{s-1}{s+1}y-y^2\right),
$$
where $\omega(y)$ is any function.
\end{lemma}


\section{The sets of invariant points}
\label{sec: Invariant points}

\begin{lemma}
The following properties of the  set of invariant points of the function $f^s_k$ are true:
\begin{itemize}
\item the set of invariant points of $f^s _k$ is a continuum set, and its Hausdorff--Besicovitch dimension is equal to  $\frac{1}{k}\log_s j$, when there exists the set  $\{\sigma_1,\sigma_2,...,\sigma_j\}$ $(j \ge 2)$ of k-digit combinations $\sigma_1,...,\sigma_j$ of s-adic digits such that
$$
\theta(a^{(i)} _1,a^{(i)} _2,...,a^{(i)} _k)=(a^{(i)} _1,a^{(i)} _2,...,a^{(i)} _k),~\mbox{where} ~\sigma_i=(a^{(i)} _1 a^{(i)} _2...a^{(i)} _k),~i=\overline{1,j};
$$

\item the  set of invariant points of $f^s_k$ is a finite set, when there exists the unique k-digit combination $\sigma$ of s-adic digits such that
$$
\theta(a_1,a_2,...,a_k)=(a_1,a_2,...,a_k),~\sigma=(a_1a_2...a_k);
$$

\item the  set of invariant points of $f^s_k$ is an empty set, when there not exist any k-digit combination $\sigma$ of s-adic digits such that
$$
\theta(a_1,a_2,...,a_k)=(a_1,a_2,...,a_k),~\sigma=(a_1a_2...a_k).
$$
\end{itemize}
\end{lemma}
\begin{proof}  From  \cite{Serbenyuk:2011:TMV}, it follows that a set of all numbers from $[0;1]$ such that the s-adic representation of numbers has only combinations from $\{\sigma_1,\sigma_2,...,\sigma_j\}$, where  $(j \ge 2)$ and $\sigma_i$, $i=\overline{1,j}$, are $k$-digit combinations  of s-adic digits, is a self-similar fractal, and its Hausdorff--Besicovitch dimension $\alpha_0$ satisfies the equation
$$
j\left(\frac{1}{s}\right)^{k\alpha_0}=1.
$$

Let there exist a unique s-adic digit combination  $\sigma$. It is obvious that the one-element set 
$$
\{x: x=\Delta^s _{(a_1a_2...a_k)(a_1a_2...a_k)(a_1a_2...a_k)...}\}
$$
is a set of invariant points of $f^s_k$.
\end{proof}

\begin{lemma}
The set of invariant points of the function $f_{+},$ as well as
$f^{-1} _{+},$ is a self-similar fractal, and its Hausdorff--Besicovitch
dimension is equal to $\frac1 2$.
\end{lemma}
\begin{proof}
  From
$$
\sum^{\infty} _{n=1} {\frac{(-1)^n \alpha_n}{s^n}}=\sum^{\infty} _{n=1} {\frac{\alpha_n}{s^n}}
$$
for any $x \in \left\{x: x=2\Delta^s _{0\alpha_20\alpha_4...0\alpha_{2n}0\alpha_{2n+2}0...}\right\},$ 
 it follows that the set $\{x:~f_{+}(x)=~x\}$ is a set of all numbers from $[0;1]$ such that the s-adic representation of numbers has only 2-digit combinations from          
$\{00,01,02,...,0(s-~1)\}$. Hence, from \cite{Serbenyuk:2011:TMV}, it is follows that the sets of invariant points of the functions $f_+$ and $f^{-1} _+$ are self-similar fractals, and their Hausdorff--Besicovitch dimension $\alpha_0$ satisfies the equation
$$
s \left (\frac{1}{s} \right)^{2\alpha_0}=1.
$$
So, $\alpha_0=\frac{1}{2}$.
\end{proof}


\section{Differential properties}
\label{sec: Differential properties}

\begin{theorem}
The function   $f \in \Lambda_{s}$ such that $f(x)\ne x$, $f(x)\ne -\frac{s-1}{s+1}-x$ and $f(x)\ne 1-x$  is continuous at s-adic irrational or nega-s-adic irrational points, and the s-adic rational or nega-s-adic rational points are points of discontinuity of the function (according to the number representation of the argument of a function $f$).
\end{theorem}
\begin{proof} Let us consider  a function of the $f^s_k$ type.  Let $x=\Delta^s _{\alpha_1\alpha_2...\alpha_n...} \in~[0;1]$ be an arbitrary number.

Let $x_0$ be an s-adic irrational number.

Then there exists $n_0=n_0(x)$ such that
$$
\left\{
\begin{array}{rcl}
\alpha_m (x)&=&\alpha_m (x_0), ~~~m=\overline{1,n_0-1};\\
\alpha_{n_0} (x)& \ne &\alpha_{n_0} (x_0).\\
\end{array}
\right.
$$
From the system, it follows that the conditions $x \to x_0$ and $n_0 \to \infty$ are equivalent, and 
$$
\left|f^s _k (x) - f^s _k (x_0)\right|=\left|\sum^{\infty} _{j=n_0} {\frac{\beta_j (x)-\beta_j (x_0)}{s^k}}\right|\le \sum^{\infty} _{j=n_0} {\frac{|\beta_j (x)-\beta_j (x_0)|}{s^k}}\le 
$$
$$
\le \sum^{\infty} _{j=n_0} {\frac{s-1}{s^k}}=\frac{1}{s^{n_0-1}}\to 0 ~\mbox{for} ~n_0 \to \infty.
$$

So, the function $f^s _k$  is continuous at s-adic irrational points. For other functions from $\Lambda_{s},$ the proof of this fact is similar.

Let $x_0=\Delta^s _{\beta_1\beta_2...\beta_n...}$ be an s-adic rational number. Then 
$$
\lim_{x \to x_0-0} {f^s _k (x)}=\Delta^s _{\gamma_1\gamma_2...\gamma_t\tau\tau\tau...},~ \mbox{where}~t\equiv 0 \pmod k,~ $$
$$
(\underbrace{\tau\tau...\tau}_{k})=\theta(\underbrace{s-1,s-1,...,s-1}_{k})
$$
and
$$
(\gamma_1,...,\gamma_k)=\theta(\beta_1,...,\beta_k), (\gamma_{k+1},...,\gamma_{2k})=\theta(\beta_{k+1},...,\beta_{2k}),..., 
$$
$$
(\gamma_{r+1},\gamma_{r+2},...,\gamma_{t})=\theta(\underbrace{\beta_{r+1},...,\beta_{n-1},(\beta_n-1),s-1,s-1,...,s-1} _{k}),~r=\left[\frac{n}{k}\right].
$$
$$
\lim_{x \to x_0+0} {f^s _k}(x)=\Delta^s _{\gamma_1\gamma_2...\gamma_r\delta_{r+1}\delta_{r+2}...\delta_{t}\xi\xi\xi...},~ \mbox{where} ~ \theta(\underbrace{0,...,0}_{k})=(\underbrace{\xi,...,\xi}_{k})
$$
and
$$
(\gamma_1,...,\gamma_k)=\theta(\beta_1,...,\beta_k),...,(\gamma_{r-k+1},...,\gamma_r)=\theta(\beta_{r-k+1},...,\beta_r), 
$$
$$
(\underbrace{\delta_{r+1},\delta_{r+2},...,\delta_t} _{k})=\theta(\underbrace{\beta_{r+1},\beta_{r+2},...,\beta_n,0,0,...,0} _{k}).
$$
So, $x_0$ is a point of discontinuity:
$$
\sum^{t} _{j=r+1} {\frac{\delta_j- \gamma_j}{s^j}}+\sum^{\infty} _{j=t+1} {\frac{\xi-\tau}{s^j}}=\sum^{t} _{j=r+1} {\frac{\delta_j- \gamma_j}{s^j}}+\frac{\xi-\tau}{(s-1)s^t}.
$$
The proof is analogous for the functions  $f_+$, $f^{-1} $, $  f_+\circ f^s _k$, $f^s _k \circ f^{-1} _+$, $f_+\circ~f^s _k\circ~f^{-1} _+$.
\end{proof}

\begin{theorem}
The function $f \in \Lambda_{s}$ such that $f$ is not $\Lambda_s$-linear function is nowhere differentiable.
\end{theorem}
\begin{proof}
Consider function $f^s _k$. Let $(x_n)$ be a  sequence  of numbers such that $x_n=\Delta^s _{\alpha_1\alpha_2...\alpha_{n-1}\alpha_n\alpha_{n+1}...} \in [0;1]$. Fix a number  $x_0=\Delta^s _{\alpha_1\alpha_2...\alpha_{n-1}c\alpha_{n+1}...}$, where $c$ is a fixed s-adic digit such that it is used infinitely many times in the s-adic representation of $x_0$. Then
$$
x_n-x_0=\frac{\alpha_n-c}{s^n},
$$
$$
f^s _k(x_n)-f^s _k(x_0)=\Delta^s _{\underbrace{0...0} _{r}[\gamma_{r+1}-\gamma^{'} _{r+1}][\gamma_{r+2}-\gamma^{'} _{r+2}]...[\gamma_{n}-\gamma^{'} _{n}]...[\gamma_{r+k}-\gamma^{'} _{r+k}]000...}, ~\mbox{where}
$$
$r=\left[\frac{n}{k}\right]$ is the integer part of $\frac{n}{k}$ and
$$
(\gamma_{r+1},\gamma_{r+2},...,\gamma_{n-1},\gamma_n,\gamma_{n+1},...,\gamma_{r+k})=\theta(\alpha_{r+1},\alpha_{r+2},...,\alpha_{n-1},\alpha_n,\alpha_{n+1},...,\alpha_{r+k}),
$$
$$
(\gamma^{'} _{r+1},\gamma^{'} _{r+2},...,\gamma^{'} _{n-1},\gamma^{'} _n,\gamma^{'} _{n+1},...,\gamma^{'} _{r+k})=\theta(\alpha_{r+1},\alpha_{r+2},...,\alpha_{n-1},c,\alpha_{n+1},...,\alpha_{r+k}).
$$
It is clear that  the conditions   $x \to x_0$ and  $n \to \infty$ are equivalent. Therefore,
$$
{f^s _k}^{'} (x_0)=\lim_{n \to \infty} {\frac{\sum^{r+k} _{i=r+1} {\frac{\gamma_i - \gamma^{'} _i}{s^i}}}{\frac{\alpha_n-c}{s^n}}}.
$$
  Since, for different combinations $\alpha_{r+1}...\alpha_n...\alpha_{r+k}$ and $\alpha_{r+1}...\alpha_{n-1}c\alpha_{n+1}...\alpha_{r+k},$ of s-adic digits, the derivative of $f^s _k$ at the point $x_0$ has different values, the function $f^s _k$ is nowhere differentiable.

The derivative of $f_k^s$ exists only if 
$$
\sum^{r+k} _{i=r+1} {\frac{\gamma_i - \gamma^{'} _i}{s^i}}=\pm \frac{\alpha_n-c}{s^n},
$$
i.e.,  when $f^s _k=x$ or $f^s _k=1-x$.

Let us consider the function $f_+$ and the notations
$$
a=\sum^{\infty} _{n=1} {\frac{\alpha_{2n-1}(x)}{s^{2n-1}}}-\sum^{\infty} _{n=1} {\frac{\alpha_{2n-1}(x_0)}{s^{2n-1}}},
$$
$$
b=\sum^{\infty} _{n=1} {\frac{\alpha_{2n}(x)}{s^{2n}}}-\sum^{\infty} _{n=1} {\frac{\alpha_{2n}(x_0)}{s^{2n}}}.
$$
Then
$$
(f_{+}(x_0))'=\lim_{x \to x_0} {\frac{f(x)-f(x_0)}{x-x_0}}=\lim_{a \to 0,b \to 0}{\frac{b-a}{a+b}}=\left[
\begin{array}{rcl}
a&=&\rho \cos(\varphi),\\
b & = &\rho \sin(\varphi),\\
& &\rho \to 0, \\
\end{array}
\right]=
$$
$$
=\lim_{\rho \to 0}{\frac{\sin(\varphi)-\cos(\varphi)}{\sin(\varphi)+\cos(\varphi)}}, \varphi=\overline{0,2\pi},
$$
i.e., the  limit does not exist. It is similar for  $f^{-1} _+$ and other functions $f \in \Lambda_{s}$.
\end{proof}


\section{The Hausdorff-Besicovitch dimension of  graphs of considered functions}
\label{sec: Hausdorff-Besicovitch dimension}

\begin{theorem}
The Hausdorff--Besicovitch dimension of the graph of any function from the class  $\Lambda_{s}$ is equal to~$1$. 
\end{theorem}
\begin{proof} From the definition and properties of the functions   $f_+$ and  $f^{-1} _+$,  it  follows that the graph of the function belongs to $s$ squares from $s^2$ first-rank squares:  
$$
\sqcap_{(ii)}=\left[\frac{i}{s};\frac{i+1}{s}\right]\times\left[-\frac{i+1}{s};-\frac{i}{s}\right],~i \in A ~~\mbox{for}~ f_+,
$$
$$
\sqcap_{(ii)}=\left[-\frac{i+1}{s};-\frac{i}{s}\right]\times\left[\frac{i}{s};\frac{i+1}{s}\right],~i \in A ~~\mbox{ for } ~f^{-1} _+,
$$
i.e., $\sqcap_{(00)},\sqcap_{(11)},\sqcap_{(22)},...,\sqcap_{((s-1)(s-1))}$.

The graph of each of the functions belongs to  $s^2$ squares  from  $s^4$ second-rank squares:
$$
\sqcap_{(i_1i_1)(i_2i_2)}=\left[\frac{i_1}{s}+\frac{i_2}{s^2};\frac{i_1}{s}+\frac{i_2+1}{s^2}\right]\times\left[-\frac{i_1}{s}+\frac{i_2}{s^2};-\frac{i_1}{3}+\frac{i_2+1}{s^2}\right] \mbox{for}~ f_+,
$$
and 
$$
\sqcap_{(i_1i_1)(i_2i_2)}=\left[-\frac{i_1}{s}+\frac{i_2}{s^2};-\frac{i_1}{3}+\frac{i_2+1}{s^2}\right]\times\left[\frac{i_1}{s}+\frac{i_2}{s^2};\frac{i_1}{s}+\frac{i_2+1}{s^2}\right] \mbox{for} ~f^{-1} _+,
$$
$i_1 \in A, i_2 \in A$, i.e.,
\begin{itemize}
\item The part of the graph, which is in the square $\sqcap_{(00)},$ belongs to $s$ squares  $\sqcap_{(00)(00)},\sqcap_{(00)(11)},  \sqcap_{(00)(22)}, ..., \sqcap_{(00)((s-1)(s-1))}$;

\item the part of the graph, which is in the square $\sqcap_{(11)},$  belongs to $s$ squares  $\sqcap_{(11)(00)},\sqcap_{(11)(11)},  \sqcap_{(11)(22)}, ..., \sqcap_{(11)((s-1)(s-1))}$;
$$
...................................................................
$$
\item the part of the graph, which is in the square $\sqcap_{((s-1)(s-1))},$  belongs to $s$ squares  $\sqcap_{((s-1)(s-1))(00)}$, $\sqcap_{((s-1)(s-1))(11)}$,  $\sqcap_{((s-1)(s-1))(22)}$, ..., $\sqcap_{((s-1)(s-1))((s-1)(s-1))}$, etc.
\end{itemize}

The graphs $\Gamma_{f_+} $ and $\Gamma_{f^{-1} _+}$ of the functions $f_+$ and $f^{-1} _+$ belong to  $s^m$ squares of rank  $m$ with  side  $s^{-m}$. Then
$$
\widehat{H}_{\alpha} (\Gamma_{f_+})=\widehat{H}_{\alpha} (\Gamma_{f^{-1} _+})=\varliminf_{m \to \infty} {s^m \left(\sqrt{s^{-2m}+s^{-2m}}\right)^{\alpha}}=\varliminf_{m \to \infty} {s^m \left(2\cdot s^{-2m}\right)^{\frac{\alpha}{2}}}=
$$
$$
=\varliminf_{m \to \infty} {\left(s^{\frac{2m}{\alpha}-2m}\cdot 2\right)^{\frac{\alpha}{2}}}=\varliminf_{m \to \infty} {\left(2^{\frac{\alpha}{2}}\cdot (s^{1-\alpha})^{m}\right)}.
$$

It is obvious that if $s^{(1-\alpha)m} \to 0$ for $\alpha >1,$ and the graphs of the functions have self-similar properties, then $\alpha^K (\Gamma_{f_+})=\alpha^K (\Gamma_{f^{-1} _+})=\alpha_0 (\Gamma_{f^{-1} _+})=~\alpha_0 (\Gamma_{f _+})=~1$, where $\alpha^K (E)$ is the fractal cell entropy dimension \cite{Pratsiovytyi:2002:FVO} of the set $E$.

To find the Hausdorff--Besicovitch dimension of the graph of the function $f^s _k,$ we use a s-adic square of the rank that is multiple to $k$.

From the definition of $f^s _k$ and properties of $f^s _k,$ it follow that  the graph of the function belongs to  $s^k$ squares from $s^{2k}$ squares of rank $k$.

The graph of  $f^s _k$ belongs to   $s^{2k}$  squares from   $s^{4k}$ squares  of rank $2k$, etc.

So, the graph $\Gamma_{f^s _k}$  of $f^s _k$  belongs to  $s^{mk}$ squares of rank $mk$ with sides  $s^{-mk}$, $m \in \mathbb N$. These squares are  $s^{2mk}$. Therefore,
$$
\hat{H}_{\alpha}(\Gamma_{f^s _k})=\varliminf_{m \to \infty} {s^{mk}\left(\sqrt{2\cdot s^{-2mk}}\right)^{\alpha}}=\varliminf_{m \to \infty}{s^{mk}\left(2^{\frac{\alpha}{2}}\cdot s^{-\alpha mk}\right)}=\varliminf_{m \to \infty}{s^{mk}{\left(2^{\frac{\alpha}{2}}\cdot s^{(1-\alpha)mk} \right)}}.
$$
It is similar with $f_+$ and $f^{-1} _+$  that $s^{(1-\alpha)mk} \to 0$ for  $\alpha>1$ and $\alpha^K (\Gamma_{f^s _k})=~\alpha_0 (\Gamma_{f^s _k})=~1$.  

The proof is analogous for other functions from $\Lambda_s$.
\end{proof}

\section{Lebesgue integral}
\label{sec:Lebesgue integral}

\begin{theorem} Let  $f \in \Lambda_{s}$. Then
$$
\int\limits_{D(f)} f(x) \, \mathrm dx=\frac{1}{2},~\mbox{where  $D(f)$ is the domain of definition of $f$.}
$$
\end{theorem}
\begin{proof}The conditions of existence of the Lebesgue integral are true for the function  $f^s_k$. Since the function has self-similar properties, the Lebesgue integral $I$  of $f^s_k$  can be calculated by the following equality: 
$$
I=\frac{1}{s^{2k}}s^k I+\frac{s^{2k}-s^k}{2}\cdot\frac{1}{s^{2k}},
$$
$$
I\left(1-\frac{1}{s^k}\right)=\frac{s^k-1}{2s^k}, I=\frac{1}{2}.
$$
It is similar for the functions $f_+$ and  $f^{-1} _+$:
$$
I=\frac{(s-1)s}{2s^2}+sI\frac{1}{s^2}=\frac{1}{s}I+\frac{s-1}{2s},
$$
whence $I=\frac{1}{2}$.

For the other functions $f \in \Lambda_s,$ one can calculate  their Lebesgue integrals with regard for the definition of the corresponding integral.
\end{proof}




\begin{thebibliography}{9}


\bibitem{Falkoner:1990:FG} K.~J.~Falconer, \emph{Fractal Geometry}, Chichester, Wiley, 1990.

\bibitem{Hensley:1992:CFC} D.~Hensley, Continued fraction Cantor  sets, Hausdorff dimension, and functional analysis, \emph{J.~Number Theory} \textbf{40}, 336--358 (1992).
  
\bibitem{Pelyuh:1974:VTF} H.~Pelyuh, A.~Sharkovskyi,  \emph{Vvedenie v teoriu functsionalnyh uravneniy} [\emph{Introduction to the Theory of Functional Equations}], Naukova dumka, Kyiv, 1974. 

\bibitem{Pratsevityi:1989:NKP} M.~Pratsiovytyi, Nepreryvnye kantorovskie proektory [Continuous Cantor  projectors], v: \emph{Metody issledovania algebraicheskih i topologicheskih struktur}, in: [\emph{Methods of Study of Algebraic and Topological Structures}], Kyiv, 95--105, 1989.

\bibitem{Pratsiovytyi:1998:FPD} M.~Pratsiovytyi, \emph{Fraktalnyi pidhid u doslidzhennjah synguliarnyh rozpodiliv} [\emph{Fractal Approach to Investigation of Singular Probability Distributions}], Vydavnytstvo NPU im. M.~P.~Dragomanova, Kyiv, 1998.  

\bibitem{Pratsiovytyi:2002:FVO} M.~Pratsiovytyi, Fractalni vlastyvosti odnieii neperervnoii nide ne dyferentsiovnoii funktsii   [Fractal properties of one  continuous nowhere differentiable function],  \emph{Naukovi Zapysky NPU im. M.~P.~Dragomanova. Phizyko-matematychni Nauky} [\emph{Trans. Natl. Pedagog. Mykhailo Dragomanov Univ. Phys. Math.}]   \textbf{3}, 351--362 (2002).  


\bibitem{Serbenyuk:2012:OMS} S.~Serbenyuk, Pro odnu maizhe skriz' neperervnu i nide ne dyferentsiovnu funktsiu, yaka zadana avtomatom zi skinchennou pamiattu   [On one nearly everywhere continuous and  nowhere differentiable function, that defined by automaton with finite memory],  \emph{Naukovyi Chasopys NPU im. M.~P.~Dragomanova. Ser.~1. Phizyko-matematychni Nauky} [\emph{Trans. Natl. Pedagog. Mykhailo Dragomanov Univ. Ser.~1. Phys. Math.}]  \textbf{13(2)}, 166--182 (2012).  

\bibitem{Serbenyuk:2011:TMV} S.~Serbenyuk, Topologo-metrychni vlastyvosti ta vykorystannya odniei uzagalnenoi mnozhyny, zadanoi s-kovym zobrazhennyam z parametrom   [Topological, metric properties
and use of one  generalized set, that  defined by s-adic representation with parameter],  \emph{Naukovyi Chasopys NPU im. M.~P.~Dragomanova. Ser.~1. Phizyko-matematychni Nauky} [\emph{Trans. Natl. Pedagog. Mykhailo Dragomanov Univ. Ser.~1. Phys. Math.}] \textbf{12}, 66--75 (2011).  

\bibitem{Turbin:1992:FMF} A.~Turbin, M.~Pratsiovytyi,   \emph{Fraktalnye mnozhestva, funkcii, rasspredeleniya}   [\emph{Fractal Sets, Functions, Probability Distributions}],  Naukova dumka, Kyiv, 1992.  

\end{thebibliography}


\end{document}